\documentclass{amsart}
\usepackage{amsfonts,amssymb,amscd,amsmath,enumerate,verbatim,newlfont,calc}
\usepackage{graphicx}
\RequirePackage[all]{xy}

\newtheorem{theorem}{Theorem}[section]
\newtheorem{theorem A}{Theorem A}
\newtheorem{theorem B}{Theorem B}
\newtheorem{lemma}[theorem]{Lemma}

\newtheorem{definition}[theorem]{Definition}

\newtheorem{proposition}[theorem]{Proposition}
\begin{document}
\authors
\title{Bogomolov multiplier and the Lazard correspondence}
\author[Z. Araghi Rostami]{Zeinab Araghi Rostami}
\author[M. Parvizi]{Mohsen Parvizi}
\author[P. Niroomand]{Peyman Niroomand}
\address{Department of Pure Mathematics\\
Ferdowsi University of Mashhad, Mashhad, Iran}
\email{araghirostami@gmail.com, zeinabaraghirostami@stu.um.ac.ir}
\address{Department of Pure Mathematics\\
Ferdowsi University of Mashhad, Mashhad, Iran}
\email{parvizi@um.ac.ir}
\address{School of Mathematics and Computer Science\\
Damghan University, Damghan, Iran}
\email{niroomand@du.ac.ir, p$\_$niroomand@yahoo.com}
\address{Department of Mathematics, Ferdowsi University of Mashhad, Mashhad, Iran}
\maketitle
\begin{abstract}
In this paper we extend the notion of CP covers for groups to the field of Lie algebras, and show that despite the case of groups, all CP covers of a Lie algebra are isomorphic. Finally we show that CP covers of groups and Lie rings which are in Lazard correspondence, are in Lazard correspondence too, and the Bogomolov multipliers are isomorphic.
\end{abstract}
\section{\bf{Introduction}}
Bogomolov multiplier and commutativity preserving cover or CP cover were first studid by Moravec for the class of finite groups (see \cite{10}, for more information). In the class of groups, the Bogomolov multiplier of a group is unique up to isomorphism but the corresponding CP cover is not necessarily unique. In the recent work \cite{1} we defined the Bogomolov multiplier of Lie algebras, and here, we will define CP covers of Lie algebras, then we will show that all CP covers of a Lie algebra are isomorphic. The Lazard correspondence that was introduced by Lazard in \cite{15}, builds an equivalence of categories between finite $p$-groups of nilpotency class at most $p-1$ and the finite $p$-Lie rings (this means that the additive group of the Lie ring is a finite $p$-group. In other words, the $p$-Lie ring is a Lie algebra over ${\Bbb{Z}}/{{p^k}{\Bbb{Z}}}$ for some positive integer $k$, see \cite{21} for more information) of the same order and nilpotency class. There is a close conection between many invariants of an arbitrary group and a Lie ring that is its Lazard correspondent. For example, the centers, the Schur multipliers and the epicenters of them are isomorphic as abelian groups (see for instance \cite{7}). Finally we will prove that if $G$ a group and $L$ be its Lazard correspondent, then the Bogomolov multipliers of them are isomorphic as abelian groups.
\section{\bf{Bogomolov multiplier and CP cover of Lie algebras}}
The section is devoted to introduce CP covers of Lie algebras and then we will show (unlike the situation in finite groups), all CP covers for a Lie algebra are isomorphic. We recall the following definition.
\begin{definition}\label{d2.1}
%{\bf{Definition 2.1.}}
Let $R$ be a commutative unital ring. A Lie algebra over $R$ is an $R$-module $L$ equipped with an $R$-bilinear map $[.,.] : L\times L \rightarrow L$ which is called the Lie bracket provided that the following conditions hold.
\begin{itemize} 
\item {$[x,x]=0$,} 
\item {$[x,[y,z]]+[z,[x,y]+[y,[z,x]]=0$ (Jaccobi identity), and $[[x,y],z]+[[y,z],x]+[[z,x],y]=0$,}
\item {$[ax+by,z]=a[x,z]+b[y,z]$ and $[z,ax+by]=a[z,x]+b[z,y]$,}
\item {$[x,y]= - [y,x],$}
\end{itemize}
for all $x,y,z \in L$ and $a,b\in R$.
\end{definition}
The Lie bracket $([x,y])$ is called the commutator of $x$ and $y$. 
\\
Note that throughout this section, $L$ will represent a Lie algebra over the field of scalars. Also, the dimension of a Lie algebra is its dimension as a vector space over field, and the cardinality of a minimal generating set of a Lie algebra is always less than or equal to its dimension. 
\\
\\
{\bf{Bogomolov multiplier.}} The Bogomolov multiplier is a group-theorical invariant introduced as an obstruction to the rationality problem in algebraic geometry. Let $K$ be a field, $G$ be a finite group and $V$ be a faithful representation of $G$ over $K$. Then there is natural action of $G$ upon the field of rational functions $K(V)$. The rationality problem (also known as Noether's problem) asks whether the field of $G$-invariant functions ${K(V)}^G$ is rational (purely transcendental) over $K$? A question related to the above mentioned is whether there exist indpendent variables $x_1,...,x_r$ such that ${K(V)}^{G}(x_1,...,x_r)$ becomes a pure transcendental extension of $K$? Saltman in \cite{22} found examples of groups of order $p^9$ with a negative answer to the Noether's problem, even when taking $K=\Bbb{C}$. His main method was the application of the unramified cohomology group ${H_{nr}^{2}}({\Bbb{C}(V)}^{G},{\Bbb{Q}}/{\Bbb{Z}})$ as an obstruction. Bogomolov in \cite{4} proved that it is canonically isomorphic to
$$B_0(G)={\bigcap}\ {\ker} \{ res_{G}^{A} : H^2(G,{\Bbb{Q}}/{\Bbb{Z}}) \rightarrow H^2(A,{\Bbb{Q}}/{\Bbb{Z}}) \},$$
where $A$ is an abelian subgroup of $G$.
The group $B_0(G)$ is a subgroup of the Schur multiplier $\mathcal{M}(G)=H^2(G,{\Bbb{Q}}/{\Bbb{Z}})$ and Kunyavskii in \cite{14} named it the \emph{Bogomolov Multiplier} of $G$. Thus vanishing the Bogomolov multiplier leads to positive answer to Noether's problem. But it's not always easy to calculate Bogomolov multipliers of groups. Moravec in \cite{19} introduced an equivalent definition of the Bogomolov multiplier. In this sense, he used a notion of the nonabelian exterior square $G\wedge G$ of a group $G$ to obtain a new description of the Bogomolov multiplier. He showed that if $G$ is a finite group, then $B_0(G)$ is non-canonically isomorphic to $\text{Hom} (\tilde{B_0}(G),{\Bbb{Q}}/{\Bbb{Z}})$, where the group $\tilde{B_0}(G)$ can be described as a section of the nonabelian exterior square of a group $G$. Also, he proved that $\tilde{B_0}(G)=\mathcal{M}(G)/\mathcal{M}_0(G)$, such that the Schur multiplier $\mathcal{M}(G)$ or the same $H^2(G,{\Bbb{Q}}/{\Bbb{Z}})$ interpreted as the kernel of the commutator homomorphism $G\wedge G \rightarrow [G,G]$ given by $x\wedge y \rightarrow  [x,y]$, and $\mathcal{M}_0(G)$ is the subgroup of $\mathcal{M}(G)$ defined as $\mathcal{M}_0(G)=<x\wedge y \ | \  [x,y]=0 , \  x,y\in G>$. Thus in the class of finite group, $\tilde{B_0}(G)$ is non-canonically isomorphic to $B_0(G)$.
With this definition all truly nontrivial nonuniversal commutator relations is collected into an abelian group that is called Bogomolov multiplier. Furthermore, Moravec's method relates Bogomolov multiplier to the concept of commuting probability of a group and shows that the Bogomolov multiplier plays an important role in commutativity preserving central extensions of groups, that are famous cases in $K$-theory. So, there are some papers to compute this multiplier for some groups. See for example \cite{4,6,12,13,14,16,17,19}. In the recent work \cite{1}, as a close relationship between groups and Lie algebras, we developed the analogus theory of commutativity-preserving exterior product and then Bogomolov multiplier for the class of Lie algebras. 
\\
\\
{\bf{Hopf-type formula for Bogomolov multiplier}} 
\\
We recall Hopf-type formula for groups and Lie algebras as follows. Let $K(F)$ denotes $\{[x,y] | x,y \in F \}$.
\begin{theorem}\label{t2.2}
%{\bf{Theorem 2.3.1.}}
Let $G$ be a group and $L$ be a Lie algebra. Then
\renewcommand {\labelenumi}{(\roman{enumi})} 
\begin{enumerate}
\item{If $G\cong \frac{F_1}{R_1}$ be a presention for $G$, then ${\tilde{B_0}}(G)\cong\frac{(R_1\cap F^{2}_1)}{<K(F_1)\cap R_1>}$},
\item{If $L\cong \frac{F_2}{R_2}$ be a presention for $L$, then ${\tilde{B_0}}(L)\cong \frac{(R_2\cap F^{2}_2)}{<K(F_2)\cap R_2>}$}.
\end{enumerate}
\end{theorem}
\begin{proof}
(i) See \cite[Proposition 3.8]{19}. (ii) from \cite{8}, $L\wedge L \cong {F^{2}_2}/[R_2,F_2]$ and $L^2\cong {F^{2}_2}/{(R_2\cap F^{2}_2)}$. Moreover the map $\tilde{\kappa} : L\wedge L \to {L^2}$ given by $x\wedge y \to [x,y]$ is an epimorphism. Thus, $\ker \tilde{\kappa}=\mathcal{M}(L)\cong (R_2\cap {F^{2}_2})/[R_2,F_2]$ and $\mathcal{M}_0(L)$ can be determined with the subalgebra of ${F_2}/{[R_2,F_2]}$ generated by all the commutators in ${F_2}/{[R_2,F_2]}$ that belong to $\mathcal{M}(L)$. Thus we have the following isomorphism for $\mathcal{M}_0(L)$,
$$<K(\dfrac{F_2}{[R_2,F_2]})\cap \dfrac{R_2}{[R_2,F_2]}> \cong \dfrac{<K(F_2)\cap R_2>+[R_2,F_2]}{[R_2,F_2]} \cong\dfrac{<K(F_2)\cap R_2>}{[R_2,F_2]}.$$
\\
Therefore ${\tilde{B_0}}(L) ={\mathcal{M}(L)}/{\mathcal{M}_0(L)}\cong {R_2\cap F^{2}_2}/{<K(F_2)\cap R_2>}$, as required.
\end{proof}
{\bf{Commutativity preserving extension of groups.}}
For groups, in parallel to the classical theory of central extension, Jezernik and Moravec in \cite{10} developed a version of extension that preserve commutativity. Let $G$, $N$ and $Q$ be groups. An exact sequence of group $1 \xrightarrow{} N \xrightarrow{\chi}  G \xrightarrow{\pi} Q \xrightarrow{} 1$, is called a comutativity preserving extension (CP extension) of $N$ by $Q$, if commuting pairs of elements of $Q$ have commuting lifts in $G$. A special type of CP extensions with the central kernel is named a central CP extension. Jezernik and Moravec in \cite{10} proved that the central exact sequence $1 \xrightarrow{} N \xrightarrow{\chi}  G \xrightarrow{\pi} Q \xrightarrow{} 1$ is a CP extension if and only if $\chi (N) \cap K(G)=1$. Also, a central CP extension $1 \xrightarrow{} N \xrightarrow{\chi} G \xrightarrow{\pi} Q \xrightarrow{} 1$ is termed stem, whenever $\chi(N)\leq G'$, and every stem central CP extension with $|N|=|\tilde{B_0}(Q)|$, is called CP cover. It is proved in \cite[Theorem 4.2]{10}, every finite group has a CP cover and every stem central CP extension is a quotient of a CP cover, and if $G$ be a CP cover of $Q$ with kernel $N$, then $N\cong \tilde{B_0}(Q)$. Now, we investigate analogue statement for Lie algebras, and then we show all CP covers for a finite dimensional Lie algebra are isomorphic.
\\
\\
The following definition is used in the next.
\begin{definition}\label{d2.3}
%{\bf{Definition 2.5.}}
Let $C$ and $\tilde{B_0}$ be Lie algebras. We call a pair of Lie algebras $(C,\tilde{B_0})$, a commutativity preserving defining pair (CP defining pair) for $L$, if
\renewcommand {\labelenumi}{(\roman{enumi})} 
\begin{enumerate}
\item{$L\cong {C}/{\tilde{B_0}}$}
\item{$\tilde{B_0} \subseteq Z(C)\cap {C^2}$}
\item{$\tilde{B_0} \cap K(C)=0$}.
\end{enumerate}
\end{definition}
In the other words, for every stem central CP extension $0 \xrightarrow{} \tilde{B_0} \xrightarrow{}  C \xrightarrow{\pi} L \xrightarrow{} 0$ with $L\cong {C}/{\tilde{B_0}}$, $(C,\tilde{B_0})$ is termed CP defining pair.
\begin{lemma}\label{l2.4}
%{\bf{Lemma 2.6.}}
Let $L$ be a Lie algebra of finite dimension $n$ and $C$ be the first term in a CP defining pair for $L$. Then $\dim C\leq n(n+1)/2$.
\end{lemma}
\begin{proof}
Clearly, $\dim {C}/{Z(C)}\leq \dim {C}/{\tilde{B_0}}=\dim L=n$. Now, if $\{x_1 +Z(C),...,x_n +Z(C)\}$ is a basis for ${C}/{Z(C)}$, then $\{[x_i,x_j] ; 1\leq i<j \leq n\}$ is a generating set for $C^2$ and since $[x_i,x_i]=0$ and $[x_i,x_j]=-[x_j,x_i]$, we have $\dim C^2\leq ({n^2}-n)/2=n(n-1)/2$. Thus $\dim {\tilde{B_0}}\leq n(n-1)/2$ and $\dim C=n+\dim {\tilde{B_0}}\leq n+n(n-1)/2=n(n+1)/2$. 
\end{proof}
A pair $(C,\tilde{B_0})$ is called a maximal CP defining pair if the dimension of $C$ is maximal.
\begin{definition}\label{d2.5}
%{\bf{Definition 2.7.}}
For a maximal CP defining pair $(C,\tilde{B_0})$, $C$ is called a commutativity preserving cover or (CP cover) for $L$. 
\end{definition}
The following definition is used for finding the Hopf-type formula for $\tilde{B_0}$, where $(C,\tilde{B_0})$ is a maximal CP defining pair, and it is used to prove the uniqueness of the CP covers of a Lie algebra.
\begin{definition}\label{d2.6}
%{\bf{Definition 2.8.}}
Let
\\
$c(L)=\{(C,\lambda) \ | \   \lambda \in \text{Hom}(C,L)\  , \    
\text{$\lambda$ surjective and} \  \ker \lambda \subseteq {C^2}\cap Z(C) \ , \  $
\\
$\ker \lambda \cap K(C)=0 \}$. \\
$(T,\tau)$ is called a universal member in $c(L)$ if for each $(C,\lambda)\in c(L)$, there exist a ${h'}\in \text{Hom}(T,C)$ such that ${\lambda}o{h'}=\tau$, in other words the following diagram commutes.
\[
\xymatrix{
T \ar[r]^{\tau} \ar[d]_{h'}& L\\
C\ar[ru]_{\lambda}&
}
\]
\end{definition}
It can be shown that, CP defining pairs and elements of $c(L)$ are related in the following sence. 
\\
Let $(C,\sigma)\in c(L)$, so $\ker \sigma \subseteq Z(C)\cap {C^2}$, $\ker\sigma \cap K(C)=0$ and $L\cong {C}/{\ker{\sigma}}$. Therefore $(C,\ker {\sigma})$ is a CP defining pair for $L$. Conversely, if $(C,N)$ is a CP defining pair for $L$, then there is a surjective homomorphism $\sigma : C\to L$ with $\ker {\sigma} = N\subseteq Z(C)\cap {C^2}$ and $N\cap K(C)=0$. Thus $(C,\sigma)\in c(L)$.
\\
\\
Now, we want to show that all CP covers  of a Lie algebra are isomorphic. First, we recall the following lemma. 
\begin{lemma}\cite[Lemma 1.4]{3}\label{l2.7}
%{\bf{Lemma 2.9.[3. Lemma 1.4]}}
Let $(N,\mu)\in c(L)$ and $\lambda \in \text{Hom}(C,L)$ where $\lambda$ is surjective. Suppose that $\beta \in \text{Hom}(C,N)$ such that ${\mu}o{\beta}=\lambda$, then $\beta$ is surjective.
\end{lemma}
We are going to show that $c(L)$ has a universal element and they are precisely those elements $(T,\tau)$ where $T$ is a CP cover of $L$.
\begin{proposition}\label{p2.8}
%{\bf{Proposition 2.10.}}
Let $L$ be a finite dimensional Lie algebra. $(T,\tau)$ is a universal element of $c(L)$ if and only if $T$ is a CP cover.
\end{proposition}
\begin{proof}
Let $(T,\tau)\in c(L)$ such that for each $(C,\lambda)\in c(L)$, there is $\rho \in \text{Hom}(T,C)$ such that ${\lambda}o{\rho}=\tau$. By using Lemma \ref{l2.7}, $\rho$ is surjective and $\dim C\leq \dim T$. Thus $T$ is a CP cover of $L$. Also, since every CP cover of $L$ is the homomorphic image of $T$ and has the same dimension as $T$, so it is isomorphic to $T$. Moreover by using Lemma \ref{l2.7}, any CP cover of $L$, gives an universal element in $c(L)$.
\end{proof}
\begin{proposition}\label{p2.9}
%{\bf{Proposition 2.11.}}
Let $L$ be a finite dimensional Lie algebra, then all CP covers of $L$ are isomorphic.
\end{proposition}
\begin{proof}
By using Proposition \ref{p2.8}, since there is a universal element in $c(L)$, all CP covers of $L$ are isomorphic.  
\end{proof}
To find the Hopf-type formula for $\tilde{B_0}$, when $(C,\tilde{B_0})$ is a maximal CP defining pair of $L$, let $L\cong F/R$ be a free presention of a finite dimensional Lie algebra $L$, $A_L={R}/{<K(F)\cap R>}$, $B_L={F}/{<K(F)\cap R>}$ and $D_L=({{F^2}\cap R})/{<K(F)\cap R>}$. We will show that there is a central ideal $E_L$ of $B_L$ complement to $D_L$ in $A_L$, also there are $\lambda$, $\tilde{\sigma}$ and $\tilde{\pi}$, such that $({B_L}/{E_L} , \overline{\pi})$ is an universal element of $c(L)$, $\overline{\sigma}\in \text{Hom}({B_L}/{E_L}, C)$ and $\overline{\pi}={\lambda}o{\overline{\sigma}}$. Also ${B_L}/{E_L}$ is a CP cover of $L$ and $({B_L}/{E_L} , \ker{\overline{\pi}})$ is a maximal CP defining pair for $L$.
\\
\\
Since $(C,\tilde{B_0})$ is a maximal CP defining pair, there is a surjective map $\lambda : C\rightarrow L$ such that $\ker{\lambda}=\tilde{B_0}\subseteq Z(C)\cap {C^2}$, ${\tilde{B_0}}\cap K(C)=0$ and $(C,\lambda)\in c(L)$. By using Lemma \ref{l2.7}, $\sigma$ is surjective. 
On the other hand, we have the following commutative diagram.
\[
\xymatrix{
F \ar[r]^{\pi} \ar[d]_{\sigma}& L\\
C\ar[ru]_{\lambda}&
}
\]
In the following lemmas, we show that $\sigma$ induces ${\sigma}_{1}\in \text{Hom}(B_L,C)$.
\begin{lemma}\label{l2.10}
%{\bf{Lemma 2.12.}}
Let $L\cong F/R$ be a free presention of $L$, then for every $x\in F$, we have $x\in R$ if and only if $\sigma(x) \in \ker \lambda$. Moreover $<K(F)\cap R>\subseteq \ker{\sigma}$, and $\sigma$ induces a surjective homomorphism ${{\sigma}_1}\in \text{Hom}(B_L,C)$ such that ${\lambda}o{{\sigma}_1}={\pi}_1$.
\end{lemma}
\begin{proof}
Let $x\in R$. Then $0=\pi(x)={\lambda}o{\sigma(x)}$. Thus $\sigma(x)\in \ker{\lambda}$. On the other hand, let $\sigma(x)\in \ker {\lambda}$, then $\lambda(\sigma(x))=0$. It implies that $\pi(x)=0$. So, $x\in \ker {\pi}=R$. Now, since $\sigma(r)\in \ker {\lambda}\subseteq Z(C)\cap{C^2}\subseteq Z(C)$ and $\sigma(f)\in C$, for all ${r_1},{r_2},r\in R$ and $f\in F$, we have $\sigma([r,f])=[\sigma(r),\sigma(f)]=0$ and $\sigma([r_1,r_2])=[\sigma(r_1),\sigma(r_2)]=0$. Thus $<K(F)\cap R>\subseteq \ker {\sigma}$. Hence $\sigma$ induces ${{\sigma}_1}\in \text{Hom}(B_L,C)$ and ${\lambda}o{{{\sigma}_1}(f+<K(F)\cap R>)}={\lambda}o{\sigma(f)}=\pi (f)={{\pi}_1}(f+<K(F)\cap R>)$. 
Therefore ${\lambda}o{{\sigma}_1}={\pi}_1$. One can see that ${\sigma}_1$ is surjective. So, we have following commutative diagrams
\[
\xymatrix{
{B_L} \ar[r]^{{\pi}_1} \ar[d]_{{\sigma}_1}& L\\
C\ar[ru]_{\lambda}&
.}
\]
\end{proof}
%Using the notations and assumptions in Lemma \ref{l2.10}, we have
\begin{lemma}\label{l2.11}
%{\bf{Lemma 2.13.}} 
Using the notations and assumptions in Lemma 2.12, we have                                                                                
\renewcommand {\labelenumi}{(\roman{enumi})}
\begin{enumerate}
\item{${{\sigma}_1}(A_L)=\ker{\lambda}$}
\item{${{\sigma}_1}(D_L)=\ker {\lambda}$}
\item{$A_L={D_L}+\ker {{\sigma}_1}$}.
\end{enumerate}
\end{lemma}
\begin{proof}
(i) Let $y\in {{\sigma}_1}(A_L)$, then $y={{\sigma}_1}(a)$ for some $a\in {A_L}$. We have $\ker{{\pi}_1}=A_L$. so, ${\lambda}o{{\sigma}_1}(a)={{\pi}_1}(a)=0$. Hence $y\in \ker{\lambda}$. Now, let $m\in \ker{\lambda}$, then there is $b\in {B_L}$ such that ${{\sigma}_1}(b)=m$, since ${\sigma}_1$ is surjective and $0=\lambda(m)={\lambda}o{{{\sigma}_1}(b)}={{\pi}_1}(b)$, $b\in \ker{{\pi}_1}=A_L$. Thus, ${{\sigma}_1}(A_L)=\ker{\lambda}$. 
\\
(ii) Clearly ${D_L}\subseteq {A_L}$, and (i) implie ${{\sigma}_1}(D_L)\subseteq \ker {\lambda}$.  Let $y\in \ker{\lambda}$. Since $(C,\lambda)\in c(L)$ and $y\in C^2=\sigma(F^2)={{\sigma}_1}({B^{2}_L})$, there is $z\in {B^{2}_L}$ such that $y={{\sigma}_1}(z)$. Since $z\in {A_L}$, $z\in {B^{2}_L}\cap {A_L}=D_L$. Hence, $\ker{\lambda}\subseteq {{\sigma}_1}(D_L)$.
\\
(iii) Let $a\in {A_L}$. by using (i) and (ii), ${{\sigma}_1}(a)\in \ker{\lambda}={{\sigma}_1}(D_L)$. Therefore there is $d\in {D_L}$ such that ${{\sigma}_1}(a)={{\sigma}_1}(d)$. So, ${{\sigma}_1}(a-d)=0$, and $a=d+e$ for some $e\in\ker{{\sigma}_1}$. Thus ${A_L}\subseteq {D_L}+\ker{{\sigma}_1}$. On the other hand, we have ${{\sigma}_1}(x)=0$, for some $x\in \ker{{\sigma}_1}$. Therefore ${{\pi}_1}(x)={\lambda}o{{\sigma}_1}(x)=0$. So, $x\in \ker{{\pi}_1}={A_L}$. Hence $\ker{{\sigma}_1}\subseteq{A_L}$. Since ${D_L}\subseteq {A_L}$, ${D_L}+\ker{{\sigma}_1} \subseteq {A_L}$. So $A_L={D_L}+\ker {{\sigma}_1}$.
\end{proof}
Note that since $A_L={D_L}+\ker {{\sigma}_1}$ and $\ker {{\sigma}_1}={\ker{\sigma}}/{<K(F)\cap R>}$, $\ker {{\sigma}_1}$ has a finite dimensional Lie subalgebra as ${({\ker {\sigma}}\cap {R\cap {F^2}})}/{<K(F)\cap R>}$. \\ Also, ${\ker {\sigma}}/{({\ker {\sigma}}\cap R\cap {F^2})} \cong {R}/{R\cap {F^2}}$ is abelian (${L}/{L^2}\cong {F}/({{R+F}^2})$ and $({{F}/{F^2}})/({{F}/({R+F^2})})\cong ({R+F^2})/{F^2} \cong  {R}/{R\cap F^2}$). Put $E_L={R}/{R\cap {F^2}}$, clearly it is a central ideal of $A_L$. Therefore $A_L$ is a central extension of $D_L$ by the abelian Lie algebra $E_L$, and this extension splits. So $A_L={D_L}\oplus {E_L}$.\\ \\
On the other hand $[R,F]\leq<K(F)\cap R>$ and $A_L$ and $D_L$ are central ideals of $B_L$. Thus, $[D_L+\ker{{\sigma}_1},B_L]=0$ and $[\ker{{\sigma}_1},B_L]=0$. Now since ${E_L}\leq{\ker{{\sigma}_1}}$, $[E_L,B_L]=0$. So, $E_L$ is a central ideal of $B_L$. Thus, ${\sigma}_1$ and ${\pi}_1$ induce $\bar{\sigma}\in \text{Hom}({B_L}/{E_L},C)$ and $\bar{\pi}\in \text{Hom}({B_L}/{E_L},L)$, respectively. Moreover the following diagram is commutative.
\[
\xymatrix{
\frac{B_L}{E_L} \ar[r]^{\bar{\pi}} \ar[d]_{\bar{\sigma}}& L\\
C\ar[ru]_{\lambda}&
}
\]
Using the previous notations, the following Lemmas show that $({B_L}/{E_L},{A_L}/{E_L})$ is a maximal CP defining pair for $L$.
\begin{lemma}\label{l2.12}
%{\bf{Lemma 2.14.}} 
Let $L$ be a finite dimensional Lie algebra, and $L\cong F/R$ for a free Lie algebra $F$. Then $({B_L}/{E_L},{A_L}/{E_L})$ is a CP defining pair for $L$, where $E_L$ is any complementary subspace to $D_L$ in $A_L$.
\end{lemma}
\begin{proof}
Since ${A_L}\subseteq Z(B_L)$ and ${A_L}/{E_L} \subseteq Z({B_L}/{E_L})$, we have
$$\dfrac{{B_L}/{E_L}}{{A_L}/{E_L}}\cong \dfrac{B_L}{A_L}\cong \dfrac{F}{R}\cong L,$$
and
$${D_L}=\dfrac{{F^2}\cap R}{<K(F)\cap R>}\subseteq \dfrac{F^2}{<K(F)\cap R>}\cong (\dfrac{F}{<K(F)\cap R>})^{2}={{B^2_L}}.$$
Hence
$$\dfrac{A_L}{E_L}\cong \dfrac{{D_L}+{E_L}}{E_L}\subseteq \dfrac{{{B^{2}_L}}+{E_L}}{E_L}=(\dfrac{B_L}{E_L})^2.$$
Thus, ${A_L}/{E_L}\subseteq Z({B_L}/{E_L})\cap {({B_L}/{E_L})^2}$ and $({A_L}/{E_L})\cap K({B_L}/{E_L})=0$.
\end{proof}
\begin{lemma}\label{l2.13}
%{\bf{Lemma 2.15.}}
${B_L}/{E_L}$ is a CP cover of $L$ and $\tilde{B_0}\cong (F^2 \cap R)/<K(F)\cap R>$.
\end{lemma}
\begin{proof}
$C$ is a CP cover of $L$, so $\dim C\geq \dim ({{B_L}/{E_L}})$. Since $\bar{\sigma}$ is surjective, $C$ is the homomorphic image of ${B_L}/{E_L}$ and $\dim C \leq \dim ({{B_L}/{E_L}})$. Therefore $\dim C = \dim ({{B_L}/{E_L}})$ and ${B_L}/{E_L}$ is a CP cover of $L$. Now by using Propositions \ref{p2.8} and \ref{p2.9}, $({B_L}/{E_L},\tilde{\pi})$ is an universal element of $c(L)$ and $C\cong {B_L}/{E_L}$. Now since ${C}/\tilde{B_0}\cong L\cong {B_L}/{A_L}$ and $D_L \cong {A_L}/{E_L}$, $\dim \tilde{B_0}= \dim {D_L}$. And so $\tilde{B_0}\cong {D_L}= ({{F^2}\cap R})/{<K(F)\cap R>}$.
\end{proof}
The following key lemma is used in the next investigation.
\begin{lemma}\cite[Lemma 1.11]{3}\label{l2.14}
%{\bf{Lemma 2.16.[3. Lemma 1.11]}}
Let $B,D,B_1,D_1$ be Lie algebras and $B\oplus D={B_1}\oplus {D_1}$. If $B\cong B_1$ and $B$ is finite dimensional, then $D\cong D_1$.
\end{lemma}
Note that since $\dim{L}=n$ and $F$ is generated by $n$ elements, $E_L$ has finite dimensional.
\begin{lemma}\label{l2.15}
%{\bf{Lemma 2.17.}}
All CP covers of $L$ are isomorphic to ${B_L}/{E_L}$.
\end{lemma}
\begin{proof}
Let $(N,\tilde{B_0})$ be a maximal CP defining pair of $L$. So there is a surjective map $\beta: N\rightarrow L$ such that $(N,\beta)\in c(L)$. ُSimilar to the previous statement, there is a central ideal ${E'_L}$ that is complementary to $D_L$ in $A_L$, and ${{\sigma}'_1}\in \text{Hom}(B_L,N)$ such that ${E'_L}\leq \ker{{{\sigma}'_1}}$ and ${\beta}o{{\sigma}'_1}={\pi}_1$. On the other hand, ${D_L}\subseteq {{B^{2}_L}}$ and ${B^{2}_L}\cap {E_L}=0$. So we can write $Z(B_L)={B^{2}_L}\cap Z(B_L)\oplus E_L \oplus A$ where $A$ is abelian and ${B^{2}_L}\cap A=0$. Thus $B_L \cong T\oplus E_L\oplus A $, where $T$ is non-abelian. Therefore $B_L\cong E_L\oplus K_L$, such that $K_L=T\oplus A$. Similarly, there is a Lie algebra ${K'_L}$ such that $B_L\cong {E'_L}\oplus {K'_L}$. Also, $E_L$ and ${E'_L}$ are abelian Lie algebras and both are of the same finite dimensional, so $E_L\cong {E'_L}$. By using Lemma \ref{l2.14}, $K_L\cong {K'_L}$. Thus ${B_L}/{E_L}\cong {B_L}/{{E'_L}}$.  
\end{proof}
Therefore we showed that for every finite dimensional Lie algebra $L$, there is a CP cover as $\frac{B_L}{E_L}$, and for every CP defining pair $(C,\tilde{B_0})$, $C$ is isomorphic to a quotient of $\frac{B_L}{E_L}$ $(C\cong \frac{B_L}{{\ker{{\sigma}_1}}} \cong \frac{{\frac{B_L}{E_L}}}{{\frac{\ker{{\sigma}_1}}{E_L}}})$. Also, since $A_L=D_L+\ker{{\sigma}_1}$ and $\tilde{B_0} \cong \frac{A_L}{\ker{{\sigma}_1}}$, $\tilde{B_0}$ is isomorphic to a quotient of Bogomolov multiplier of $L$.
\section{\bf{Bogomolov multiplier and the Lazard correspondence}}
The section is devoted to show the Bogomolov multiplier of a Lie ring $L$ and a group $G$ is isomorphic, when $L$ is Lazard correspondent of $G$. We recall the following definition.
\begin{definition}\label{d3.1}
%{\bf{Definition 3.1.}}
A Lie ring is defined as an abelian group $L$ equipped with a $\Bbb{Z}$-bilinear map $[.,.]: L\times L \rightarrow L$ called the Lie bracket satisfying the following conditions 
\begin{itemize} 
\item {$[x,x]=0$,} 
\item {$[x,[y,z]]+[z,[x,y]+[y,[z,x]]=0$ and $[[x,y],z]+[[y,z],x]+[[z,x],y]=0$ (Jaccobi identity),}
\item {$[x+y,z]=[x,z]+[y,z]$ and $[x,y+z]=[x,y]+[x,z]$,}
\item {$[x,y]= - [y,x]$,}
\end{itemize}
for all $x,y,z\in L$.
\end{definition}
The Lie bracket $([x,y])$ is called the commutator of $x$ and $y$.
\\
%Subrings, ideals, homomorphisms, free Lie rings and (finite) presentions of Lie rings are defined as usual (see [---], for more informations).
%The ring $L$ is a free Lie ring if there exist only those relations between its elements which are consequences of the postulates. It follows that a free Lie ring is uniquely determined by devoting a set of free generators $x_1,...,x_n$, and that it consists of all formal polynomials $L(x_1,...,x_n)$ in these generators. 
\\
Let $L$ be a Lie ring and $M$ and $N$ are subrings of it, we define $[M,N]$ as the Lie subring of $L$ generated by all commutators $[m,n]$ with $m\in M$ and $n\in N$. This allows us to define the lower central series $L = L^1 \geq L^2 \geq L^3 ≥ ...$ via $L^i = [L^{i-1}, L]$. The Lie ring $L$ is nilpotent if this series terminates at ${0}$, and in this case, the class $cl(L)$ is the length of the lower central series of $L$.
\\
\\
Note that a Lie ring which is also an algebra over a field (or a commutative unital ring) is termed a Lie algebra over that field (or commutative unital ring). Also a Lie ring can be defined as a $\Bbb{Z}$-Lie algebra (see \cite{9}), and $p$-Lie ring is a Lie algebra over ${\Bbb{Z}}/{{p^k}{\Bbb{Z}}}$ for some positive integer $k$ (see \cite{21}). Therefore more definitions and proofs of Lie rings can be obtained as generalizations from the Lie algebras, and there are similar results between finite Lie rings and finite dimensional Lie algebras over a field. So similar to recent work \cite{1}, we have the Bogomolov multiplier for Lie rings. Also we want to introduce CP defining pairs and CP covers of Lie rings.
\\
\\
In the following, for a given Lie ring $L$, the set $\{[x,y] | x,y \in L \}$ of all commutators of $L$ is denoted by $K(L)$. Also, the notations are the same as notations in the previous section.
\begin{definition}\label{d3.2}
%{\bf{Definition 3.2.}}
Let $C$ and $K$ be finite Lie rings, a pair $(C,K)$ is called a commutativity preserving defining pair (CP defining pair) for $L$, provided that 
$L\cong {C}/{K}$, $K \subseteq Z(C)\cap {C^2}$ and $K \cap K(C)=0$.
\end{definition}
\begin{lemma}\label{l3.3}
%{\bf{Lemma 3.3.}}
Let $L$ be a Lie ring of finite order $n$ and $C$ be the first term in a CP defining pair for $L$. Then $| C |\leq {n^2}(n-1)$.
\end{lemma}
\begin{proof}
Clearly, $| {C}/{Z(C)} | \leq | {C}/K |=| L |=n$. So $| K | \leq | {C^2} | \leq n(n-1)$. Since $| C |=| L | | K |$, $| C |\leq {n^2}(n-1)$. 
\end{proof}
Therefore if $L$ is a finite Lie ring, the order of the first coardinate of CP defining pairs for $L$ are bounded, and a pair $(C,K)$ is called a maximal CP defining pair, if the order of $C$ is maximal.
\begin{definition}\label{d3.4}
%{\bf{Definition 3.4.}}
For a maximal CP defining pair $(C,K)$, $C$ is called a commutativity preserving cover or (CP cover) for $L$. 
\end{definition}
Similar to the proofs in the previous section, it can be proven that for every finite Lie ring $L$, there is a unique CP cover, and for every CP defining pair $(C,K)$, $C$ and $K$ are isomorphic to a quotient of CP cover and Bogomolov multiplier of $L$, respectively. Also for every maximal CP defining pair as $(C,K)$, we have the following Lemma.
\begin{lemma}\label{l3.5}
%{\bf{Lemma 3.5.}}
Let $(C,K)$ be a maximal CP defining pair for finite Lie ring $L$. Then $K$ is isomorphic to the Bogomolov multiplier of $L$.
\end{lemma}
\begin{proof}
Similar to the Lemma \ref{l2.13} and by using the previous assumptions, $|D_L|=|K|$. Since $\lambda o \sigma=\pi$, $\sigma (F)\in K=\ker {\lambda}$ is equivalent to $F\in R$. Also
$$\bar{\sigma}(D_L)=\sigma (R\cap {F^2})=\sigma (R) \cap \sigma (F^2)=K\cap {C^2}=K.$$
So, ${\bar{\sigma}}|_{D_L}: {D_L}\rightarrow {K}$ is a surjective, and since $|D_L|=|K|$, $K\cong {D_L}$.
\end{proof}
 In the following, we introduce Lazard correspondence between finite $p$-Lie rings of nilpotency class at most $p-1$ and the finite $p$-groups of the same order and nilpotency class. Also, to avoid confusion, in a group $G$, we denote the multiplication by $xy$ and the commutator is used to $[x,y]_{G}=x^{-1}y^{-1}xy$. 
\\
\\
{\bf{The Baker-Campbell-Hausdorff formula (B-C-H) and its inverse.}} Let $L$ be a $p$-Lie ring of order $p^n$ and nilpotency class $c$ with $p-1\geq c$ and $G$ be a finite $p$-group with order $p^n$ and the nilpotency class $c\leq p-1$. The Baker-Campbell-Hausdorff formula (B-C-H formula) is a group multiplication in terms of Lie ring operations
$$xy := x+y+{\frac{1}{2}[x,y]_{L}+\frac{1}{12}[x,x,y]_{L}+...},$$
where $x,y\in L$. The inverse $g^{-1}$ of the group element $g$ corresponds to $-g$. and the identity $1$ in the group corresponds to $0$ in the Lie ring. So, the B-C-H formula is used to turn Lie ring presentions into group presentions. Conversely the inverse B-C-H formula is a Lie ring addition and Lie bracket in terms of group multiplication that it is used to turn group presentions into Lie ring presentions. This have the general form
$$x+y := xy{[x,y]_{G}^{\frac{-1}{2}}. ...}$$
$$[x,y]_{L} := [x,y]_{G}{[x,x,y]_{G}^{\frac{1}{2}}. ...}$$
See \cite{5}, when $c\leq 14$.
\\
\\
{\bf{The Lazard correspondence.}} The B-C-H formula and it's inverse give an isomorphism between the category of nilpotent $p$-Lie rings of order $p^n$ and the nilpotency class $c$, provided $p-1\geq c$ and the category of finite $p$-groups of the same order and nilpotency class that it is known as the Lazard correspondence. By using this correspondence, in the same line of investigation, the same results on $p$-groups can be checked on $p$-Lie rings. In the following, we mention some of these correspondences that were proved by Eick in \cite{7}.
\begin{proposition}\cite[Proposition 3]{7}\label{p3.8}
%{\bf{Proposition 3.9.1.[7. Proposition 3]}}
Let $G$ be a finite $p$-group of class at most $p-1$, and $L$ be its Lazard correspondent. Let $X$ be a subset of $G$ and hence of $L$. Then
\renewcommand {\labelenumi}{(\roman{enumi})} 
\begin{enumerate}
\item{There is a Lazard correspondence between the subring ${L_0}\subseteq L$ generated by $X$ and the subgroup ${G_0}\subseteq G$ generated by $X$}.
\item{$L_0$ is an ideal of $L$ if and only if $G_0$ is a normal subgroup of $G$}.
\end{enumerate}
\end{proposition}
\begin{proposition}\cite[Proposition 4]{7}\label{p3.9}
%{\bf{Proposition 3.9.2.[7. Proposition 4]}}
Let $G$ be a finite $p$-group of class at most $p-1$, and $L$ be its Lazard correspondent. Then
\renewcommand {\labelenumi}{(\roman{enumi})} 
\begin{enumerate}
\item{$Z(G)$ and $Z(L)$ coincide as sets and are isomorphic as abelian groups}.
\item{$G'$ and $L^2$ coincide as sets and are in Lazard correspondence}.
\end{enumerate}
\end{proposition}
\begin{proposition}\cite[Proposition 5]{7}\label{p3.10}
%{\bf{Proposition 3.9.3.[7. Proposition 5]}}
Let $G$ be a finite $p$-group of class at most $p-1$, and $L$ be its Lazard correspondent. Let $G_0$ be a normal subgroup in $G$ and $L_0$ be the correponding ideal in $L$. Then $\psi : {G}/{G_0} \to {L}/{L_0}$ given by $(xG_0 \longmapsto x+L_0)$, is a well-defined bijection, and it induces a Lazard correspondence between ${G}/{G_0}$ and ${L}/{L_0}$.
\end{proposition}
Note that similar to the CP cover definition of groups in \cite{10}, for every stem central CP extension $1 \xrightarrow{} N \xrightarrow{}  G \xrightarrow{\pi} Q \xrightarrow{} 1$ with $Q\cong G/N$, $(N,G)$ is termed CP defining pair of $Q$, and $G$ is called CP cover, whenever $|N|=|\tilde{B_0}(Q)|$.
\begin{proposition}\label{p3.11}
%{\bf{Proposition 3.10.}}
Let $G$ be a finite $p$-group of class at most $p-1$, and $L$ be its Lazard correspondent. Then every CP defining pair of $G$ is in the Lazard correspondence with a CP defining pair of $L$ and vice versa.
\end{proposition}
\begin{proof}
Suppose $(G^*,G_0)$ is a CP defining pair of $G$ with $G\cong {G^*}/{G_0}$. Then $cl(G^*)<cl(G)+1\leq p-1+1=p$. So, $cl(G^*)\leq p-1$ and there is a Lie ring $L^*$ in the Lazard correspondence with $G^*$. Since $G_0 \subseteq (G^*)'\cap Z(G^*)$ and $G_0 \cap K(G^*)=1$, by Proposition \ref{p3.8} and the Lazard correspondence, there is a central ideal $L_0$ of $L^*$ such that $G_0$ and $L_0$ are in the Lazard correspondence, $L_0\subseteq (L^*)^2 \cap Z(L^*)$ and $L_0 \cap K(L^*)=0$. Now, Proposition \ref{p3.10} shows that $G\cong {G^*}/{G_0}$ and ${L^*}/{L_0}$ are in the Lazard correspondence. So, $L\cong {L^*}/{L_0}$. Hence, $(L^*,L_0)$ is a unique CP defining pair of $L$. And the converse prove follow similarly.
\end{proof}
We know that there are many invariants between $G$ and its Lazard correspondent. Now, we want to introduce another instance of these invariants.
\begin{theorem}\label{3.12}
%{\bf{Theorem 3.11.}}
Let $G$ be a finite $p$-group of class at most $p-1$, and $L$ be its Lazard correspondent. Then
\renewcommand {\labelenumi}{(\roman{enumi})} 
\begin{enumerate}
\item{The isomorphism types of CP covers of $G$ are in  the Lazard correspondence with the isomorphism types of CP covers of $L$ and vice versa}.
\item{The Bogomolov multipliers of $G$ and $L$ are isomorphic as abelian groups}.
\end{enumerate}
\end{theorem}
\begin{proof}
Given that CP covers and Bogomolov multipliers are the first and second components of the maximal CP defining pairs of groups and Lie rings, respectively, this is a direct consequence of the Proposition \ref{p3.11} and the B-C-H formula.
\end{proof}
{\bf{Example 3.11.}}
We consider a finite $p$-group $G_{1p}$ of order $p^{5}$ with the nilpotency class $3$ and the following presention
$$G_{1p}=<g,g_1,g_2,g_3\ |\   [g_1,g]=g_2 , [g_2,g]=g^{p^{2}}=g_3 , g^{p}_1=g^{p}_2=g^{p}_3=1>$$
Moravec in \cite{19} showed that $\tilde{B_0}(G_{1p})=0$. For $p\geq 5$ these groups are in the Lazard correspondence with finite $p$-Lie ring $L_{1p}$ of the same order and nilpotency class. For fixed prime $p$, the method of \cite{7} can be used to determine the Lie ring presention for $L_{1p}$ with $p$ as parameter. Let $F_{1p}$ be a free Lie ring on $v,v_1,v_2,v_3$, and denote presentations of $L_{1p}$ as ${F_{1p}}/{R_{1p}}$. So,
$$L_{1p}=<v,v_1,v_2,v_3 \ |\  [v_1,v]=v_{2}-{{p^2}v}/{2} ={v_2}-{v_3}/{2}, [v_2,v]={p^{2}}v+{{p^4}v}/{2}=v_3,$$
$$ p{v_1}=p{v_2}=p{v_3}=0>.$$
The Moravec's method in \cite{20} may used to determine the Bogomolov multiplier for a polycyclic group translates to finite $p$-Lie ring. Now we use this method to determine the Bogomolov multiplier of $L_{1p}$. Based on the above presention, we have
$$L_{1p}\wedge L_{1p}=<v\wedge {v_1} , v\wedge {v_2} , v\wedge {v_3} , {v_1}\wedge {v_2} , {v_1}\wedge {v_3} , {v_2}\wedge {v_3}>.$$
Hence for all $w\in \mathcal{M}(L_{1p})\leq L_{1p}\wedge L_{1p}$ there exist ${\alpha}_1,\ldots,{\alpha}_6 \in {\Bbb{Z}}_{p^k}$, such that $w={{\alpha}_1}(v\wedge {v_1})+{{\alpha}_2}(v\wedge {v_2})+{{\alpha}_3}(v\wedge {v_3})+{{\alpha}_4}({v_1}\wedge {v_2})+{{\alpha}_5}({v_1}\wedge {v_3})+{{\alpha}_6}({v_2}\wedge {v_3})$.
\\
Consider $\tilde{\kappa} : L_{1p}\wedge L_{1p} \to {L^{2}_{1p}}$. $\tilde{\kappa}(w)=0$, and
$${{\alpha}_1}[v,{v_1}]+{{\alpha}_2}[v,{v_2}]+{{\alpha}_3}[v,{v_3}]+{{\alpha}_4}[{v_1},{v_2}]+{{\alpha}_5}[{v_1},{v_3}]+{{\alpha}_6}[{v_2},{v_3}]=0.$$
So, ${{\alpha}_1}{v_2}+({{\alpha}_2}-{{\alpha}_1}/{2}){v_3}=0$. Hence ${{\alpha}_1}={{\alpha}_2}=0$. Therefore
$$w={{\alpha}_3}(v\wedge {v_3})+{{\alpha}_4}({v_1}\wedge {v_2})+{{\alpha}_5}({v_1}\wedge {v_3})+{{\alpha}_6}({v_2}\wedge {v_3}).$$
On the other hand, $[v,v_3]=[v_1,v_2]=[v_1,v_3]=[v_2,v_3]=0$. Thus
\\
 $(v\wedge {v_3}) , ({v_1}\wedge {v_2}) , ({v_1}\wedge {v_3}) , ({v_2}\wedge {v_3}) \in \mathcal{M}_0(L_{1p})$. So, $w\in \mathcal{M}_0(L_{1p})$ and
 \\
  $\mathcal{M}(L_{1p})\subseteq \mathcal{M}_0(L_{1p})$. Hence $\tilde{B_0}(L_{1p})=0$.
  \\
  \\
{\bf{Example 3.12.}}
We consider a finite $p$-group $G_{2p}$ of order $p^{6}$ with the nilpotency class $3$ given the following presention
$$G_{2p}=<g,g_1,g_2,g_3,g_4,g_5\ |\  [g_1,g_2]=g_3 , [g_3,g_1]=g_4 , [g_3,g_2]=g_5 , [g,g_1]=g_4,$$ $${g^{p}_1}={g^{p}_2}={g^{p}_3}={g^{p}_4}={g^{p}_5}={g}^p=1>$$
Chen and Ma in \cite{6} showed that $\tilde{B_0}(G_{2p})=0$. For $p\geq 5$ these groups are in the Lazard correspondence with the finite $p$-Lie ring $L_{2p}$ of the same order and nilpotency class. For fixed prime $p$, the method of \cite{7} can be used to determine the Lie ring presention for $L_{2p}$ with $p$ as parameter. Let $F_{2p}$ be a free Lie ring on $v,v_1,\ldots,v_5$, and denote presentation of $L_{2p}$ as ${F_{2p}}/{R_{2p}}$. So,
$$L_{2p}=<v,v_1,v_2,v_3,v_4,v_5 \ |\ [v_1,v_2]=v_{3}-{v_4}/{2}-{v_5}/{2} , [v_3,v_1]=v_4 , [v_3,v_2]=v_5 ,$$
$$[v,v_1]=v_4 , p{v_1}=p{v_2}=p{v_3}=p{v_4}=p{v_5}=0>.$$
By Moravec's method to determine the Bogomolov multiplier for a polycyclic group translate to finite $p$-Lie ring. We use this to determine the Bogomolov multiplier of $L_{2p}$. Based on the above presention we have
$$L_{2p}\wedge L_{2p} = <v\wedge {v_1},v\wedge {v_2},v\wedge {v_3},v\wedge {v_4},v\wedge {v_5},{v_1}\wedge {v_2},$$
$${v_1}\wedge {v_3},{v_1}\wedge {v_4},{v_1}\wedge {v_5},{v_2}\wedge {v_3},{v_2}\wedge {v_4},{v_2}\wedge {v_5},{v_3}\wedge {v_4},{v_3}\wedge {v_5},{v_4}\wedge {v_5}>.$$
For all $w\in M(L_{2p})\leq L_{2p}\wedge L_{2p}$, there exists ${{\alpha}_1},\ldots ,{{\alpha}_{15}} \in {\Bbb{Z}}_{p^k}$, such that
$$w={{\alpha}_1}(v\wedge {v_1})+{{\alpha}_2}(v\wedge {v_2})+{{\alpha}_3}(v\wedge {v_3})+{{\alpha}_4}(v\wedge {v_4})+{{\alpha}_5}(v\wedge {v_5})+{{\alpha}_6}({v_1}\wedge {v_2})+{{\alpha}_7}({v_1}\wedge {v_3})$$
$$+{{\alpha}_8}({v_1}\wedge {v_4})+{{\alpha}_9}({v_1}\wedge {v_5})+{{\alpha}_{10}}({v_2}\wedge {v_3})+{{\alpha}_{11}}({v_2}\wedge {v_4})+{{\alpha}_{12}}({v_2}\wedge {v_5})+{{\alpha}_{13}}({v_3}\wedge {v_4})$$
$$+{{\alpha}_{14}}({v_3}\wedge {v_5})+{{\alpha}_{15}}({v_4}\wedge {v_5})$$
Let $\tilde{\kappa} : L_{2p}\wedge L_{2p} \to {{L^{2}_{2p}}}$. Then we have $\tilde{\kappa}(w)=0$ and
$${{\alpha}_1}[v,{v_1}]+{{\alpha}_2}[v,{v_2}]+{{\alpha}_3}[v,{v_3}]+{{\alpha}_4}[v,{v_4}]+{{\alpha}_5}[v,{v_5}]+{{\alpha}_6}[{v_1},{v_2}]+{{\alpha}_7}[{v_1},{v_3}]$$
$$+{{\alpha}_8}[{v_1},{v_4}]+{{\alpha}_9}[{v_1}, {v_5}]+{{\alpha}_{10}}[{v_2},{v_3}]+{{\alpha}_{11}}[{v_2},{v_4}]+{{\alpha}_{12}}[{v_2},{v_5}]+{{\alpha}_{13}}[{v_3},{v_4}]+{{\alpha}_{14}}[{v_3},{v_5}]$$
$$+{{\alpha}_{15}}[{v_4},{v_5}]=0,$$
So, $({{\alpha}_1}-{{\alpha}_6}/{2}-{{\alpha}_7}){v_4}+{{\alpha}_6}{v_3}+(-{{\alpha}_6}/{2}-{{\alpha}_{10}}){v_5}=0$. Thus, ${\alpha}_6={\alpha}_{10}=0$ and ${\alpha}_1={\alpha}_7$. Therefore
$$w={{\alpha}_1}((v\wedge {v_1})+({v_1}\wedge {v_3}))+{{\alpha}_2}(v\wedge {v_2})+{{\alpha}_3}(v\wedge {v_3})+{{\alpha}_4}(v\wedge {v_4})+{{\alpha}_5}(v\wedge {v_5})+{{\alpha}_8}({v_1}\wedge {v_4})+$$
$${{\alpha}_9}({v_1}\wedge {v_5})+{{\alpha}_{11}}({v_2}\wedge {v_4})+{{\alpha}_{12}}({v_2}\wedge {v_5})+{{\alpha}_{13}}({v_3}\wedge {v_4})+{{\alpha}_{14}}({v_3}\wedge {v_5})+{{\alpha}_{15}}({v_4}\wedge {v_5})$$
On the other hand, $(v\wedge {v_2})$ , $(v\wedge {v_3})$, $(v\wedge {v_4})$ , $(v\wedge {v_5})$ , $({v_1}\wedge {v_4})$ , $({v_1}\wedge {v_5})$ , $({v_2}\wedge {v_4})$ , $({v_2}\wedge {v_5})$ , $({v_3}\wedge {v_4})$ , $({v_3}\wedge {v_5})$ , $({v_4}\wedge {v_5})$ belong to the $M_0(L_{2p})$. We can see that, $[v+{v_1} , {v_1}+{v_3}]=0$. So, $(v+{v_1})\wedge ({v_1}+{v_3}) \in \mathcal{M}_0(L_{2p})$. Hence $(v\wedge {v_1})+(v\wedge {v_3})+({v_1}\wedge {v_1})+({v_1}\wedge {v_3}) \in \mathcal{M}_0(L_{2p})$. Thus, $((v\wedge {v_1})+({v_1}\wedge {v_3})) \in \mathcal{M}_0(L_{2p})$ and $w \in \mathcal{M}_0(L_{2p})$. Therefore $\mathcal{M}(L_{2p}) \subseteq \mathcal{M}_0(L_{2p})$ and $\tilde{B_0}(L_{2p})=0$, as required.
\\
\\

\end{document}